\documentclass[letterpaper, 10 pt, conference]{ieeeconf}  

\IEEEoverridecommandlockouts                              

\usepackage{graphics} 
\usepackage{epsfig} 
\usepackage{amsmath} 
\usepackage{amsmath} 
\usepackage{amssymb}  
\usepackage{cite}
\usepackage{fullpage}
\usepackage{fancyhdr}
\usepackage{color, subfigure,multirow}

\newtheorem{lem}{Lemma}

\newenvironment{sketch-proof}{{\it Sketch of the Proof:\ }}{ \hfill}

\newcommand{\R}{{\mathbb R}}

\renewcommand{\Pr}{{\mathbb{P}}}

\newcommand{\diag}{{\text{diag}}}

\newcommand{\cov}{{\text{cov}}}

\newcommand{\figref}[1]{Figure~\ref{#1}}

\title{On Projection-Based Model Reduction of Biochemical Networks\\ Part II: The Stochastic Case}

\author{Aivar Sootla$^{1}$ and James Anderson$^{2}$
\thanks{$^{1}$AS is with the Department of Bioengineering, Imperial College London, UK {\tt\small a.sootla@imperial.ac.uk}}%
\thanks{$^{2}$JA is with St. John's College and the Department of Engineering Science, University of Oxford, Parks Road, OX1 3JP {\tt\small james.anderson@eng.ox.ac.uk}}%
}

\begin{document}
\maketitle
\thispagestyle{empty}
\pagestyle{empty}
\begin{abstract}
In this paper, we consider the problem of model order reduction of stochastic biochemical networks. In particular, we reduce the order of (the number of equations in) the Linear Noise Approximation of the Chemical Master Equation, which is often used to describe biochemical networks. In contrast to other biochemical network reduction methods, the presented one is projection-based. Projection-based methods are powerful tools, but the cost of their use is the loss of physical interpretation of the nodes in the network. In order alleviate this drawback, we employ structured projectors, which means that some nodes in the network will keep their physical interpretation. For many models in engineering, finding structured projectors is not always feasible; however, in the context of biochemical networks it is much more likely as the networks are often (almost) monotonic. To summarise, the method can serve as a trade-off between approximation quality and physical interpretation, which is illustrated on numerical examples. 
\end{abstract}
\section{Introduction and Preliminaries}
In mathematical biology, one of the most common approaches to model reduction of deterministic dynamical systems is to apply Tikhonov's theorem~\cite{Tikhonov1952}, this is also referred to as time-scale separation or the quasi-steady state assumption. The framework based on time-scale separation has received considerable attention in the control theory community~\cite{kokotovic1987singular}. However, projection-based methods (see, e.g.~\cite{Moore1981, Scherpen93}) have now become the preferred method. This shift occurred because the latter proved itself more powerful than the former. In any projection-based method first an appropriate state-space transformation is computed and only then tools similar to time-scale separation are applied. Having the freedom to apply such transformations results in the following advantages: existence of error bounds (in some cases), automatic selection of reduced states, higher approximation quality. 

A technique similar to time-scale separation have been also developed for Stochastic Differential Equations (SDEs) and known as the called stochastic averaging principle~\cite{khas1968aver_c}.In recent work~\cite{hartmann2011balanced} the averaging principle was derived for linear SDEs, while using a projection approach~\cite{Moore1981}.

In the context of  biochemical networks, we are dealing not with an SDE, but with a partial differential equation of probability distributions, the so-called Chemical Master Equation (CME). A solution to this equation is a continuous-time Markov chain with an infinite number of states. To the authors' best knowledge there are no tractable time-scale separation or averaging methods applicable to CMEs. Hence reduced order modelling is typically applied to different approximations of the CME such as Finite State Projection~\cite{munsky:044104, pahlajani2011stochastic}, Linear Noise Approximation (LNA)~\cite{thomas2012rigorous} or macroscopic deterministic models of reaction networks~\cite{rao2012model,sunnaaker2011method}. 
As in~\cite{thomas2012rigorous}, we will compute reduced order models for the LNA, which is a collection of two systems: a deterministic one describing macroscopic concentrations and an SDE driven by a Gaussian noise describing the fluctuations about the macroscopic concentrations. Our approach to reduction is a combination of averaging and projection. The focus of this paper is towards the stochastic averaging methods, in Part I \cite{SooA14} the details of the structured projectors are given and proofs included\footnote{Whilst both papers are self contained the first paper is available on line at {\color{blue}http://arxiv.org/abs/1403.3579}}.

Applying a state-space transformation results in the states of the reduced order model which are composed of linear combinations of the states of the full order model. Hence the new states lack physical interpretation and sparsity patterns in the network are lost, which constitutes a major drawback of the projection-based methods. In order to limit the loss of physical interpretation, we use structured model order reduction~\cite{SandStrucRed}. This means that the states are partitioned into two groups: one group of states is left intact, and the other states are linear combinations of some original states. The transformations are computed according to the partitioning based on the covariance matrix of fluctuations about the macroscopic concentrations of the model. This covariance matrix approximates the statistics of the solution to the CME under certain conditions~\cite{wallace2012linear}. Therefore, the reduced order model can potentially capture additional information about the stochastic nature of the model in comparison with other approaches.

The idea of using structured transformations in order to preserve the physical meaning of the states is not new. Moreover, the class of models, for which such transformations can be computed, is not rich and not many necessary conditions for their existence are known. However, in the context of biochemical networks for the systems with monotone dynamics (for definition see \cite{angeli2003monotone}) such transformations can always be computed. 

In this paper, we notice that the proposed method can be applied to a more general class of networks. Many biochemical networks have a tendency to be nearly monotone~\cite{sontag2007monotone}, that is if a small number of edges are removed, then the network becomes monotone. (In Part I \cite{SooA14} we formally define the concept of monotonicity.)
Since the monotonicity implies the existence of structured projectors, one can make a conjecture on the existence of such projectors for a large class of biochemical networks. To illustrate this observation, we apply the method to a model of yeast glycolysis~\cite{van2012testing}, which becomes monotone after removing just three reactions (out of more than thirty). For this model, structured projectors exist for any state-space partitioning, which is a very strong property.

The paper is organised as follows. In the next subsections, we introduce the CME, LNA and stochastic averaging. In Section~\ref{s:mor}, the proposed model reduction method is described, which is validated on numerical examples in Section~\ref{s:ex}.
\subsection{Modelling Biochemical Networks}
Biochemical networks are typically modelled by a continuous time infinite Markov chain, probability distribution function of which is computed by a Chemical Master Equation (CME):
\begin{equation}
  \label{eq:master}
  \frac{\partial \Pr(n, t)}{\partial t} = \Omega \sum_{i=1}^R  (\hat f (n - S_{i}, \Omega) - \hat f(n, \Omega)) \Pr(n, t),
\end{equation}
where $R$ is the number of reactions; 
column vectors $S_{i}$ form a stoichiometry matrix $S$; $\hat f$ is a vector containing the reaction rates $\hat f_i$; $n$ is a vector containing the number of molecules $n_j$ of species $j$; $\Omega$ is a volume of a compartment where reactions are occurring; finally, $\Pr(n,t)$ is the probability of the vector of the number of molecules equal to $n$ at time $t$. This equation cannot be solved analytically except for a handful of cases and numerical simulations are extremely expensive. In order to lower the complexity of simulations, different approximations of a CME are often derived, for example, the Linear Noise Approximation (LNA)~\cite{van2007stochastic}. 
%
The major assumption in the LNA is as follows:
\[
\frac{n}{\Omega} = x + \Omega^{-1/2} \eta,
\]
where $x$ is a vector of macroscopic concentrations of the species, $\eta$ is a vector of stochastic fluctuations about $x$. Additionally, if we assume a large number of reactions occurring per unit time, it can be shown, that the fluctuations $\eta$ and macroscopic concentrations $x$ obey the following equations: 
\begin{align}
\label{eq:lna-fla}   &\dot \eta = J(x) \eta + \Omega^{-1/2} S F(x) \Gamma, \\
\label{eq:det-dyn}   &\dot x = S f(x), 
\end{align}
where $f(x)$ is approximately equal to $\hat f(n,\Omega)$ for a large volume $\Omega$, $\Gamma$ is a Gaussian noise, $J$ is a Jacobian of $S f(x)$, and $F = {\diag}(\sqrt{f(x)}))$. Note that the matrices $J$, $F$ do not depend on the fluctuations $\eta$, but depend only on the macroscopic concentrations $x$, which is computed using deterministic differential equations. The covariance, $X$, of $\eta$ is computed by solving (or simulating) the Lyapunov differential equation~\cite{grima2010effective}:
\begin{equation}\label{eq:cov-lna}
  J X + X J^T + \Omega^{-1} S F^2 S = \frac{d X}{d t}.
\end{equation}
As a final remark, note that the solution to \eqref{eq:lna-fla},\eqref{eq:det-dyn} is a vector with elements $y_i(t)$, which at every time $t$ are normally distributed with mean equal to $x_i(t)$, variance equal to $X_{i i}(t)$ and covariances $\cov(y_i(t) y_j(t))$ equal to $X_{i j}(t)$~\cite{wallace2012linear}. 

The macroscopic reaction rates are not affected at all by the fluctuations dynamics, hence the time-scale separation can be applied directly to \eqref{eq:det-dyn}. In order to reduce the order of the fluctuation dynamics, we apply a version of the averaging principle. The result from~\cite{thomas2012rigorous} is presented here for completeness. Consider the system:
\begin{equation}\label{eq:fastslow}
  \begin{aligned}
   \dot \eta_s&= J_{s s} \eta_s + J_{s f} \eta_f + S_s F \Gamma, \\
   \varepsilon \dot \eta_f&= J_{f s} \eta_s + J_{f f} \eta_f + S_f F \Gamma,     
  \end{aligned}
\end{equation}
where 
\[
J = \begin{pmatrix} J_{s s} & J_{f s} \\ J_{s f} & J_{f f} \end{pmatrix} \quad
S = \begin{pmatrix} S_s \\ S_f \end{pmatrix}
\]
Note that the $s$ and $f$ subscripts denote slow and fast states respectively. Then the reduced order linear noise approximation of the fluctuations corresponding to the slow dynamics $\eta_s$ can be obtained as follows:
\begin{equation}
\label{eq:lna-fla_red}
   \dot \eta_s = (J_{s s} - J_{s f} J_{f f}^{-1} J_{f s}) \eta_f + (S_s - J_{s f}J_{f f}^{-1} S_f) F\Gamma .
\end{equation}
 As a final remark, we note that the authors~\cite{thomas2012rigorous} did not explicitly use a Tikhonov-like theorem, but derived the \eqref{eq:lna-fla_red} using projection operator theory applied to a corresponding Focker-Plank equation. This can serve as a justification for the use of time-scale separation techniques in the context of stochastic differential equations driven by white noise as in \eqref{eq:lna-fla_red}. 

\section{Model Reduction Methods \label{s:mor}}

In this paper, it is proposed to employ stochastic averaging coupled with a projection approach in order to obtain a reduced order model. First, a particular state-space transformation $T$ will be applied to species $n$ in CME~\eqref{eq:master} resulting in the new species $m$ such that $m = T n$. The transformation $T$ is computed based on the covariance matrix of $\eta$. 
Finally, we simply apply averaging as described above in order to reduce the fast species in $m$ and obtain approximate dynamics of the macroscopic concentration of species and fluctuations about the macroscopic concentrations. Additionally, we assume that the fast species in $m$ converge to a unique stationary distribution, otherwise we need to consider so called double averaging~\cite{wainrib2012double}.
\subsection{CME under a State-Space Transformation}
First, we show that  based on the solution of the CME in the species $m$, it is possible to reconstruct the solution in the species $n$. 
\begin{lem}
Assume the variable $n$ satisfies the Master equation:
\[
 \frac{\partial \Pr(n, t)}{\partial t} = \Omega \sum_{j=1}^R (\hat f (n - S_{i}, \Omega) - \hat f(n, \Omega)) \Pr(n, t).
\]
Then the variable $m = T n$ satisfies the following equation:
\[   
 \frac{\partial \Pr(m, t)}{\partial t} = \Omega \sum_{j=1}^R(\tilde f (m - \tilde S_{i}, \Omega) - \tilde f(m, \Omega)) \Pr(m, t),
\]
where $T$ is an invertible constant matrix, $\tilde S= T S$, $\tilde f(m, \Omega) = f(T^{-1} m, \Omega)$.  
\end{lem}
\begin{proof}
Consider the term $(\hat f (n - S_{i}, \Omega) - \hat f(n, \Omega))$ and make a substitution $n = T^{-1} m$. Now
\begin{multline*}
\hat f (n - S_{i}, \Omega) - \hat f(n, \Omega) =   \hat f (T^{-1} m - S_{i}, \Omega) - \\
\hat f(T^{-1} m, \Omega) = \hat f (T^{-1} ( m - T S_{i}), \Omega) - \\
\hat f(T^{-1} m, \Omega) = \tilde f ( m - \tilde S_{i}, \Omega) - \tilde f(m, \Omega)
\end{multline*}
Finally, in order to prove the claim, note that $\Pr(n, t)$ is equal to $\Pr(m,t)$.
\end{proof}
In the next section we derive the structured transformation matrix $T$.
\subsection{Structured Transformations}
 In order to compute a transformation $T$ we require an output $y$ as a function of $x$ according to classical methods~\cite{Moore1981, Scherpen93}. The output in our case can be chosen simply as observations of the states of particular interest. Assume that we are interested in the behaviour of the first $l$ species, then define the matrix $C:=\begin{pmatrix} I_{l} & 0_{l, k} \end{pmatrix}$, where $l+k$ is the total number of species. Now the outputs are chosen as follows: the output of the macroscopic concentrations $y^d$ is equal to $C x$, and the output for the fluctuations $y^f$ is equal to $C \eta$.


We propose to compute a transformation $T$ based on the covariances in a steady-state $x_{s s}$ of \eqref{eq:det-dyn}. Hence, the obtained approximation will be around the steady-state $x_{s s}$. One can also employ techniques for linear time-varying model reduction, such as \cite{san+04}, \cite{nil09dis}. Consider a system 
\begin{equation}
  \label{eq:sys}
  \begin{aligned}
    \dot \nu &=A \nu+B \Gamma\\
         y^f&= C \nu
  \end{aligned}
\end{equation}
where the drift matrix $A = J(x_{s s})$ is asymptotically stable and $B =  \Omega^{-1/2}S F(x_{s s})$.  Assume also that the system \eqref{eq:sys} is partitioned as follows
\begin{equation} \small
  \label{eq:part}
\nu = \begin{pmatrix}  \nu_1 \\ \nu_2 \end{pmatrix}\; 
A =\begin{pmatrix} A_{11} & A_{12} \\ A_{21} & A_{22} \end{pmatrix}\;
B = \begin{pmatrix} B_{1} \\ B_{2} \end{pmatrix}\; 
C^T = \begin{pmatrix}  I_{l} \\ 0_{k,l} \end{pmatrix},
\end{equation}
where $\nu_1 \in {\R}^{l}$, $\nu_2 \in {\R}^{k}$, and the matrices $A$, $B$ and $C$ are partitioned according to the same dimensions. The next step is to compute structured Gramians, which are obtained as solutions to Lyapunov inequalities
\begin{equation}
  \label{eq:lyap_ineq}
  \begin{aligned}
    A P + P A^T + B B^T&\le 0 \\
    Q A + A^T Q + C^T C&\le 0
  \end{aligned}
\end{equation}
subject to the same partitioning as the states:
\begin{equation}
\label{eq:gram_part}
P =\begin{pmatrix} P_{11} & 0_{l,k}  \\ 0_{k, l}  & P_{22} \end{pmatrix} \quad
Q =\begin{pmatrix} Q_{11} & 0_{l,k}  \\ 0_{k, l}  & Q_{22} \end{pmatrix}
\end{equation}
Note that $P$ is an approximation of $X$ from \eqref{eq:cov-lna} around the steady-state $x_{s s}$. Given our assumption about the importance of the first $l$ species, the transformation $T$ is composed as follows:
\[
   T = \begin{pmatrix} I_{l} & 0_{l,k} \\ 0_{l,k} & T_{22} \end{pmatrix}
\]
where $T_{22}$ is such that 
\[
T_{22}^{-1} P_{22} T_{22}^{-T}=T_{22}^T Q_{22} T_{22} =\Sigma_{22},
\]
and $\Sigma_{22}$ is diagonal. According to standard tools~\cite{SandStrucRed}, we choose the states to truncate according to the magnitude of the values of the diagonal of $\Sigma_{22}$. Assume $r$ states are to be reduced, let $W_{2 2}$ be the first $k-r$ columns of $T_{2 2}$, while $W_{2 2}^r$ are the rest $r$ columns of $T_{2 2}$. Let also $V_{2 2}$ be the first $k-r$ columns of $T_{22}^{-1}$, while $V_{2 2}^r$ are the rest $r$ columns of $T_{2 2}^{-1}$. Now, the projectors can be obtained as follows
{ \begin{equation}
  \label{eq:proj}
  \begin{gathered}
   W = \begin{pmatrix} I_{l} & 0_{l,k-r} \\ 0_{k-r, l} & W_{22} \end{pmatrix} \quad   W_r = \begin{pmatrix} 0_{l,r} \\  W_{22}^r \end{pmatrix} \\
   V = \begin{pmatrix} I_{l} & 0_{l,k-r} \\ 0_{k-r, l}  & V_{22} \end{pmatrix} \quad   V_r = \begin{pmatrix} 0_{l,r} \\ V_{22}^r \end{pmatrix}
  \end{gathered}
\end{equation}}
  The existence of block-diagonal solutions to \eqref{eq:gram_part} cannot be guaranteed for general linear systems. It is known, however, that some classes of systems admit block-diagonal solutions to the Lyapunov inequalities. One such class is positive systems, that is systems with $A$ matrix with non-negative off-diagonal entries (these matrices are called Metzler), $B$ and $C$ matrices with non-negative entries. A generalisation of positive systems to the non-linear case are monotone systems. This essentially implies that for monotone systems such Gramians always exist (cf.~\cite{monred2014}). We discuss this further in Part I \cite{SooA14}.

\subsection{Reduced Order Model}
Let $z$ be a new variable equal to $T x$, let also $z_r$ be the species to be removed from the model, and $z_s$ the states of the reduced order model. Now the equations approximating the full order dynamics (\ref{eq:lna-fla}-\ref{eq:det-dyn}) can be computed as follows
\begin{equation}
  \label{eq:red-met-1}
   \begin{aligned}
   \dot \eta_m&= V^T J(W z_m + W_r z_r) W \eta_m + \\
   &\qquad \Omega^{-1/2} V^T S F(W z_m + W_r z_r) W \Gamma_m, \\
   \dot z_m& = V^T S f(W z_m + W_r z_r)  \\
   0& = V_r^T S f(W z_m + W_r z_r) \\
   y_r^d& =  C (W z_m + W_r z_r)\\
   y_r^f& =  C W \eta_m.
 \end{aligned}
\end{equation}
\subsection{Model Comparison}
We assume that the full order model is defined as follows:
\[
 \begin{array}{ll}
   \dot x = S f(x),&    y^{d} = C x \\
   \dot \eta = J(x)\eta + \Omega^{-1/2} S F(x) \Gamma,  &
   y^{f} = C \eta \\
   x(0) = x_0, \quad \eta(0) = 0 & 
 \end{array} 
\]
where $y^d$ is ``the output'' of the macroscopic concentrations, $y^f$ is ``the output'' of the fluctuations with a constant matrix $C$ and $x_0$ is the initial state.
The reduced order models can be written as follows:
\[
 \begin{array}{ll}
   \dot z_m = S_r f_r(z_m),  &   y_r^{d} = C_r z + D_r \\
   \dot \eta_m = J_r(z_m) \eta_m + \Omega^{-1/2} S_r F_r(z_m) \Gamma, &  y_{r}^f = C_r \eta \\
  z_m(0) = V x_0, \quad \eta_m(0) = 0 & 
 \end{array} 
\]
where $S_r = V^T S$, $f_r(z) = f(W z + W_r z_r)$, $F_r(z) = F(W z + W_r z_r)$, $J_r(z) =  V^T J(W z + W_r z_r) W$, $C_r = C W$ and $D_r = C W_r z_r$. 
Note that~\cite{thomas2012rigorous} fits the framework in \eqref{eq:red-met-1} with the identity transformation $T$, and projectors $W$ and $V$ reducing particular states. We will take~\cite{thomas2012rigorous} as a baseline and compare it to the proposed method. 

We compare separately the error in the macroscopic dynamics (mean) and the fluctuations (variance), since their dynamic models are decoupled. The error $y^d-y^d_r$ in macroscopic dynamics is computed by perturbing the initial state $x_0$ from the steady-state $x_{s s}$ and measured in $L_1$, $L_2$ and $L_{\infty}$ norms (which for completeness we now define):
\begin{eqnarray*}
\|u\|_p& =& \left(\int_{-\infty}^{\infty} |u(t)|_p^p dt  \right)^{\frac{1}{p}}\quad \text{for }p\in \{1,2\},\\
\|u\|_{\infty} &=& \text{ess sup}_t |u(t)|_{\infty}.
\end{eqnarray*}
A comparison in terms of the fluctuations $\eta$ is performed by computing the covariance matrix of the outputs $y$ and $y_r$. For the full order model this matrix is computed as 
\[
\cov^f = \cov(y^f (y^f)^T) = C \cov(\eta \eta^T) C^T =  C X C^T
\]
where $X$ satisfies the Lyapunov equation~\eqref{eq:cov-lna}.
Similarly, the covariance matrix for the reduced order models $\cov_r^f$ can be computed.

\section{Examples \label{s:ex}}

\subsection{Toy Example. \label{ex:cov}}
 \begin{figure}[t]
    \centering
   \subfigure[Depiction of the network]{\includegraphics[width=0.33\columnwidth]{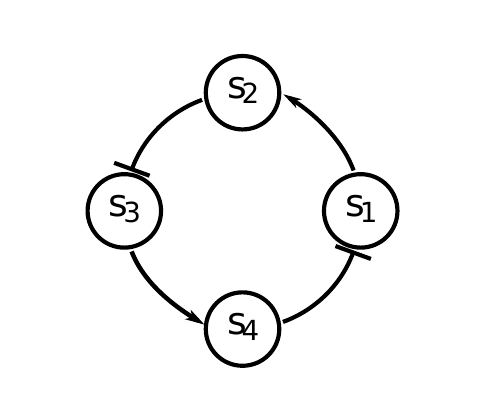} \label{fig:toy-ex}}\qquad
   \subfigure[Configuration for reduction]{\includegraphics[width=0.33\columnwidth]{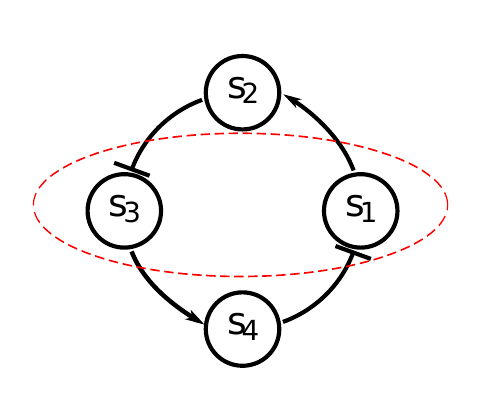} \label{fig:toy-ex-1}}
\caption{Toy Example. In the left panel, the network is schematically depicted. In the right panel, the configuration for reduction is depicted. In this configuration species $S_1$ and $S_3$ are grouped together, while reducing one state}\label{fig:1}
\end{figure}
The first network we consider consists of only four species, see \figref{fig:toy-ex}.  One can interpret the species $S_1$ and $S_3$ as mRNA, and $S_2$ and $S_4$ as the corresponding proteins. We refer the reader to Part I \cite{SooA14} for a full interpretation of the model.

   \begin{align*}
      &\dot m_i = \frac{c_{i 1}}{1 + p_j^2}-c_{i 2} m_i\\
      &\dot p_i = c_{i 3} m_i-c_{i 4} p_i
   \end{align*}
 where $c_{i1}$ are constants, $m_i$ are mRNA concentrations, $p_i$ are protein concentrations and  $i,j\in \{1,2\}$ and $i\neq j$. We compare the simulation results for the full order model, the reduced order model obtained by~\cite{thomas2012rigorous}, and the reduced order model obtained from reduction according to the configuration in \figref{fig:toy-ex-1} with parameters
\[
  c_{1\cdot}= c_{2\cdot}  = \begin{pmatrix} 3 & 4 & 1 & 0.2  \end{pmatrix}.
\]
The method from~\cite{thomas2012rigorous} and the presented method produce very similar deterministic models, simulation of which is depicted in \figref{fig:means}. But the computation of the covariance matrix of the fluctuations $\eta$ paints a different picture. The presented method provided the fluctuations with statistics very close the full order model statistics, which is not the case for the statistics of the model obtained by~\cite{thomas2012rigorous}, see \figref{fig:stoch}. 
The presented reduction method uses the covariance matrices in order to compute the reduced order models. Moreover, the proposed method targets for reduction the species with small variances in the fluctuations about the macroscopic species concentrations. This outlines a major advantage of the proposed method.  

\begin{figure}
\centering
\subfigure[The number of species $S_1$ and $S_3$ according to macroscopic reaction equations. ]{\includegraphics[width=0.45\columnwidth]{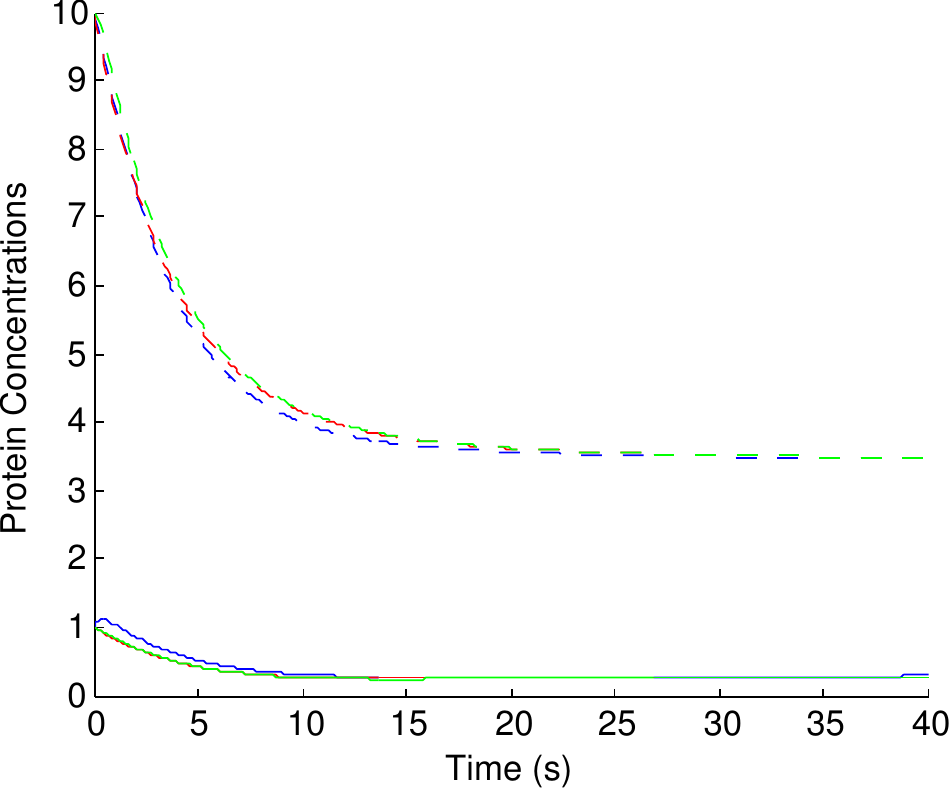} \label{fig:means}}
 \subfigure[The variance of fluctuations in the number of species $S_1$. ]{\includegraphics[width=0.45\columnwidth]{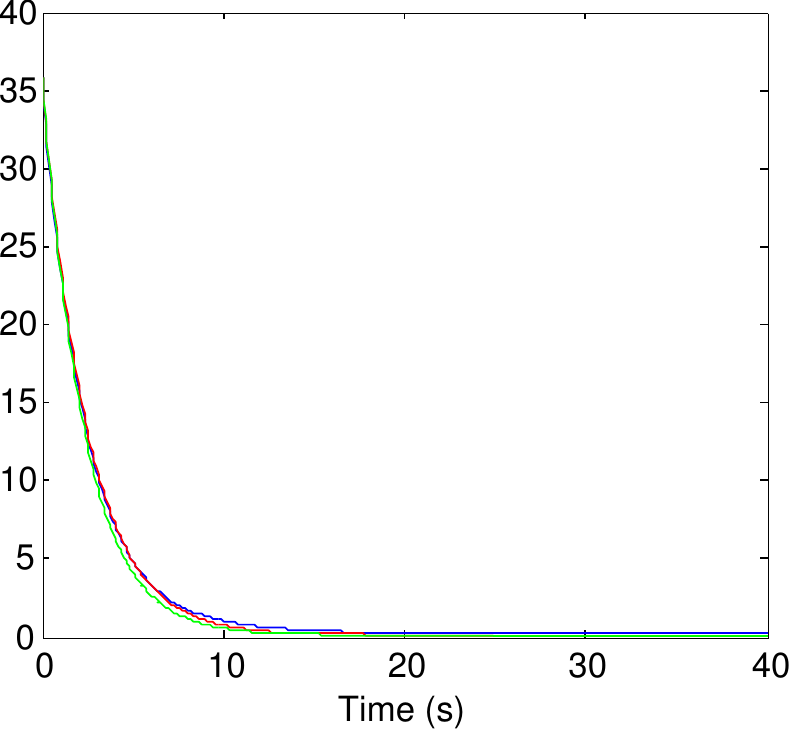} \label{fig:cov11}}

 \subfigure[The covariance in fluctuations in the number of species $S_1$ and $S_3$]{\includegraphics[width=0.45\columnwidth]{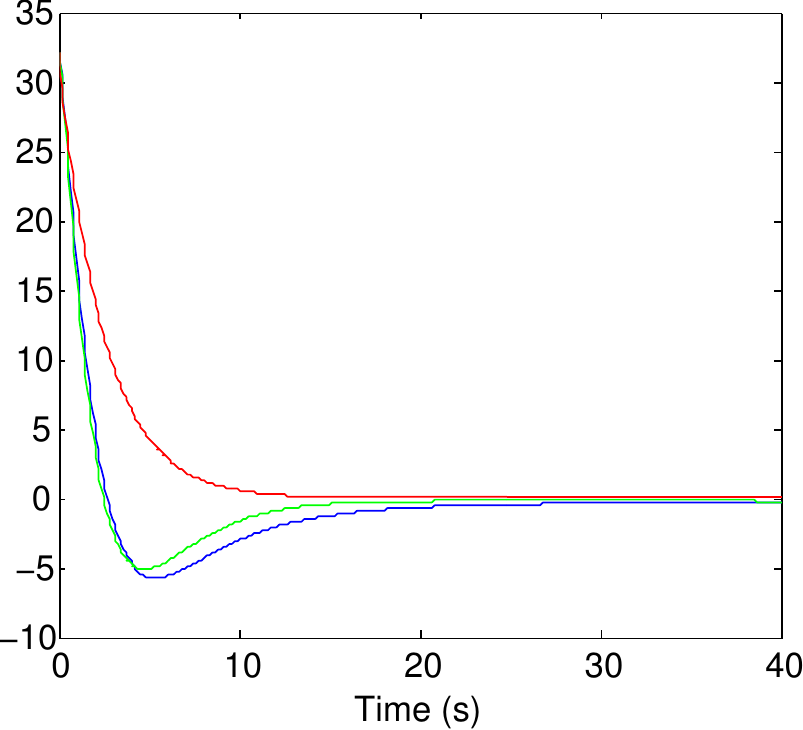} \label{fig:cov21}}
 \subfigure[The variance of fluctuations in the number of species $S_3$]{\includegraphics[width=0.45\columnwidth]{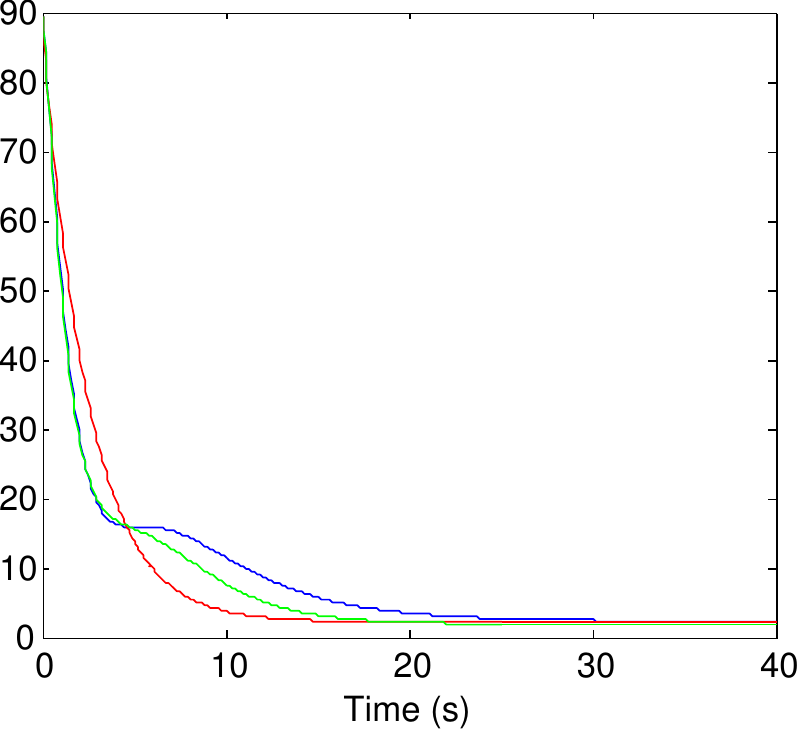} \label{fig:cov22}}

\caption{The number of species $S_1$ and $S_3$ according to macroscopic (deterministic) reaction equations and the covariance matrix of these species. In all figures, the blue lines are obtained by simulating the full order model, the red lines by simulating the model obtained using~\cite{thomas2012rigorous}, and the green lines by simulating the reduced order model with states lumped according to the configuration in \figref{fig:toy-ex-1}. In Figure~\ref{fig:means}  The dashed lines are the number of species of $S_1$, and the solid lines are the number of species of $S_3$. Even though macroscopic concentrations of species are almost the same for all the models, the covariances of the model obtained by~\cite{thomas2012rigorous} are quite different from the full order model covariances.}
\label{fig:stoch}

\end{figure}

\subsection{Kinetic Model of Yeast Glycolysis. Non-Monotone Dynamics \label{ex:gly}}
The model was published in~\cite{van2012testing} and consists of twelve metabolites and four boundary fluxes. We model the network's response to change of glucose in the system as in~\cite{rao2012model}. Again we refer the reader to Part I \cite{SooA14} for a full model description, we point out that the Jacobian of the system is not Metzler but by knocking out one uni-directional and one bi-directional reaction the network is monotone.


Using this fact, it was not a great surprise that a linearised model around a steady-state would have block-diagonal Gramians with a sparsity pattern according to some state partitioning. However, the existence of diagonal Gramians was surprising. This meant that without any reservation we could approximate any group of states, while preserving the other states. 

The simulation results are presented in Table~\ref{tab:gly-red} for various reduction configurations. We compare only the errors in the macroscopic concentrations. We apply~\cite{thomas2012rigorous} to metabolite concentrations, while using the proposed method we try to lump those metabolites in one state, so that the number of reduced states is similar in both cases. The first three rows of each sub-table in Table~\ref{tab:gly-red} can be compared directly, and it is clear that the proposed method performs better in terms of quality than~\cite{thomas2012rigorous}. 

The proposed method is also more flexible in terms of reduction choices. In the forth row of Table~\ref{tab:gly-red}-B, the region \{BPG-PEP\} contains four metabolites; however, we reduced only two states after computing the state-space transformation. In the fifth row, in the region \{GLCi-F6P\}, which contains three metabolites, we reduce just one state and this provides us with the best model among all the reduction attempts. 


\begin{figure}
  \centering
  \includegraphics[width=0.8\columnwidth]{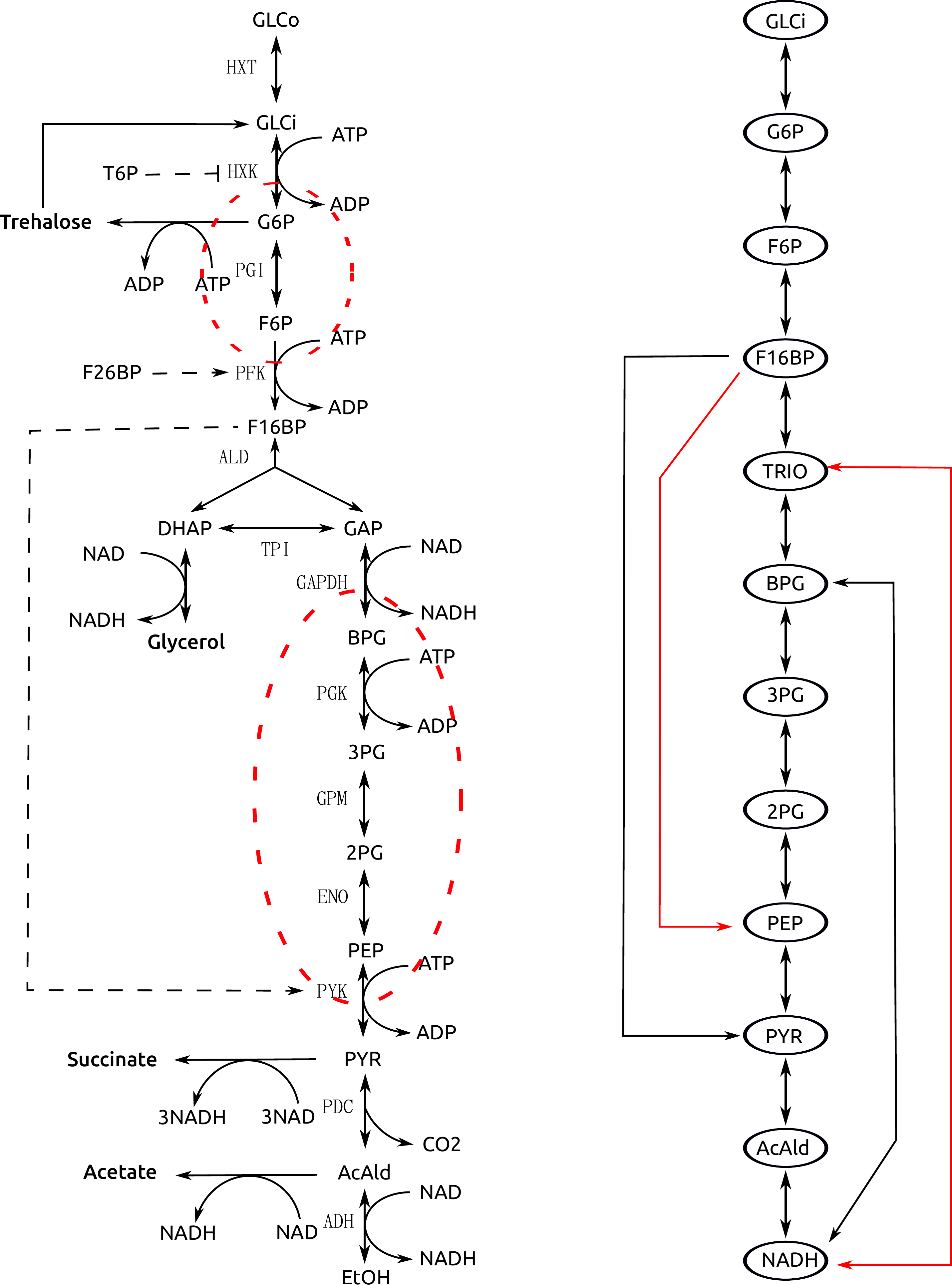} 
  \caption{Kinetic model of yeast glycolysis. In the left panel the biochemical graph is depicted. In the right panel a graph of dynamic interactions between metabolites. If the red connections are removed, the dynamics of the network would become monotone.}\label{fig:glycerol}
\end{figure}
 
\begin{table}
\centering\scriptsize
\caption{Reduction of the glycolysis model. The error of the output is given in different norms. 
} \label{tab:gly-red} 
  \begin{tabular}{cccc}
  \multicolumn{4}{c}{ \sc Table~\ref{tab:gly-red}-A. Appxoimation results obtained by using~\cite{thomas2012rigorous}} \\[6pt]
  States $\backslash$ Error & $L_1$ & $L_2$ & $L_{\infty}$\\
\hline
\hline
  F6P, 2PG, PEP              & $1.21$ & $23.6$ & $0.98$ \\
  F6P, 3PG, 2PG, PEP         & $1.56$ & $33.6$ & $1.55$ \\
  G6P, F6P, 3PG, 2PG, PEP    & $2.05$ & $36.3$ & $1.59$ \\
\hline
  \end{tabular}

 \vspace{6pt}

  \begin{tabular}{ccccc}
  \multicolumn{5}{c}{ \sc Table~\ref{tab:gly-red}-B. Reduction by $\{k_1, k_2\}$ states in every region} \\[6pt]
    Lumped Region(s)         & $\{k_1, k_2\}$  & $L_1$  & $L_2$  & $L_{\infty}$ \\
  \hline
  \hline
  \{G6P, F6P\}, \{2PG-PEP\} & $\{1, 2\}$ & $1.17$ & $24.8$ & $1.01$ \\
  \{G6P, F6P\}, \{BPG-PEP\} & $\{1, 3\}$ & $1.41$ & $22.9$ & $0.86$ \\
  \{GLCi-F6P\}, \{BPG-PEP\} & $\{2, 3\}$ & $1.23$ & $18.7$ & $0.64$ \\
  \{G6P, F6P\}, \{BPG-PEP\} & $\{1, 2\}$ & $1.40$ & $22.8$ & $0.86$ \\
  \{GLCi-F6P\}, \{BPG-PEP\} & $\{1, 2\}$ & $0.03$ & $0.56$ & $0.02$ \\
\hline
  \end{tabular}
\end{table}

\section{Conclusion and Discussion}
This paper provides a proof of concept of using projection-based methods for modelling stochastic biological systems. We illustrate on numerical examples that the proposed method is competitive with time-scale separation in terms of approximation error. In combination with Part 1 \cite{SooA14} we have shown how to construct the structured projectors and shown that for monotone dynamical system we can preserve structure, and (locally) monotonicity.

An additional inquiry is required into existence of structured Gramians for the biochemical networks. In general, diagonal Gramians do not exist for all models; however, block-diagonal ones may be more common. In any case, even the initial intuition provides a great justification for the proposed reduction method. 

Another interesting topic is an extension to non-constant transformations $T$ (non-linear or time-varying). A non-linear structured balancing can be potentially applied to a Chemical Langevin Equation (CLE), which is a more representative approximation of CME in comparison to LNA~\cite{wallace2012linear}. A non-linear transformation $T$ can also potentially address the problem of non-stationary fast variables and double averaging.

An important topic for the proposed model reduction method is partitioning of a network. Some intuition can be gained through biological insights, for example, one can collapse a whole pathway into a couple of states. One can also employ tools from metabolic control analysis~\cite{klipp2008systems}. On the other hand, we can employ purely theoretical tools such as~\cite{delvenne2013stability,sezer1986nested,anderson2012decomposition}. Currently, we are not investigating the problem of partitioning; however, it is an important future work direction.

Finally, an issue to consider is the definition of simplicity of the resulting model. In~\cite{Rewienski}, it is pointed out that many methods, which reduce the order of the model, actually result in a stiff system, which is harder to simulate. 

\section*{Acknowledgment}
The authors would like thank Prof Bayu Jayawardhana and Dr Shodhan Rao for kindly providing the kinetic model of yeast glycolisis. JA acknowledges funding through a junior research fellowship from St. John’s College, Oxford. AS is supported by the EPSRC Science and Innovation Award EP/G036004/1. 
\bibliography{Biblio}
  \end{document}